\documentclass[notitlepage,12pt,reqno]{amsart}  

\usepackage {ulem}

  \usepackage[utf8]{inputenc}

\usepackage{amsfonts, amsmath, amssymb, amsthm, bbm, color, enumerate, graphicx, mathtools, tikz, hyperref, relsize,enumitem}
\usepackage[numbers,sort]{natbib}
\usepackage{euscript}
\usepackage{comment}
\usepackage[margin=1in]{geometry}
\usepackage{cleveref}

\theoremstyle{plain}

\usepackage{color}

\definecolor{vry}{RGB}{253, 231, 37}

\definecolor{vrg}{RGB}{94,201,98}
 
\definecolor{vrdg}{RGB}{33, 145, 140}

\definecolor{vrb}{RGB}{59,82,139}

\definecolor{vrp}{RGB}{68,1,84}

\definecolor{vro}{RGB}{249,142,9}

\definecolor{vrr}{RGB}{188,55,84}

\definecolor{vrnb}{RGB}{13,8,135}

\allowdisplaybreaks

\theoremstyle{plain}

\newtheorem{theorem}{Theorem}[section]
\newtheorem{corollary}[theorem]{Corollary}
\newtheorem{lemma}[theorem]{Lemma}
\newtheorem{proposition}[theorem]{Proposition}
\theoremstyle{remark}
\newtheorem{remark}[theorem]{Remark}

\newtheorem{definition}{Definition}[section]

\newtheorem{assumption}{Assumption}[section]

\let\plainqed\qedsymbol
\newcommand{\claimqed}{$\lrcorner$}

\hypersetup{
	colorlinks=true,
	linkcolor=vrb,          
	citecolor=vrr,       
	filecolor=vrr,   
	urlcolor=vrr, 
	pdftitle={},
	pdfauthor={},
	pdfsubject={},
	pdfkeywords={},
	linktocpage=true
}

\usepackage{fancyhdr}

\usepackage{natbib}

\usepackage{amssymb}
\usepackage{amsmath}
\usepackage{amsthm}

\newcommand{\D}{{\mathrm d}}

\newcommand{\NN}{{\mathbb N}}

\newcommand{\RR}{{\mathbb R}}
\newcommand{\R}{{\mathbb R}}

\newcommand{\Z}{{\mathbb Z}}

\newcommand{\PP}{{\mathbb P}}

\newcommand{\calA}{{\mathcal A}}

\newcommand{\calD}{{\mathcal D}}

\newcommand{\calF}{{\mathcal F}}

\newcommand{\calG}{{\mathcal G}}
\newcommand{\calH}{{\mathcal H}}

\newcommand{\calK}{{\mathcal K}}

\newcommand{\calP}{{\mathcal P}}

\newcommand{\calR}{{\mathcal R}}

\newcommand{\calS}{{\mathcal S}}

\newcommand{\calU}{{\mathcal U}}

\newcommand{\calY}{{\mathcal Y}}

\newcommand{\calZ}{{\mathcal Z}}

\newcommand{\frC}{\mathfrak{C}}
\newcommand{\frD}{\mathfrak{D}}
\newcommand{\frB}{\mathfrak{B}}

\newcommand{\be}{\begin{equation}}
	\newcommand{\ee}{\end{equation}}

\numberwithin{equation}{section}

\title[Small noise asymptotics of diffusions with heavy tails]{Small noise asymptotics for a class of jump-diffusions with Heavy tails for large times}
\author[Sumith Reddy]{Sumith Reddy Anugu}

\address{Sumith Reddy Anugu\\
Institut f\"ur Mathematik\\
Technische Universit\"at Ilmenau\\
 Germany 98693 \\
}
\email{sumith-reddy.anugu@tu-ilmenau.de}

\author[Siva R. Athreya]{Siva R. Athreya}
\address{ Siva R. Athreya\\
International Centre for Theoretical Sciences\\
Bengaluru- 560089, India\\
}
\email{athreya@icts.res.in} 

\author[Vivek S. Borkar]{ Vivek S. Borkar}
\address{ Vivek S. Borkar\\
	Dept.\  of Electrical Engg.\\
	Indian Institute of Technology Bombay, Powai\\%
	Mumbai 400076, India. }
\email{borkar.vs@gmail.com}
\keywords{large time asymptotics, small noise asymptotics, weak large deviation principle,  diffusions, $\alpha$-stable process, Brownian motion, mixed continuous and impulse control problem}
\subjclass[2020]{Primary: 60F10, Secondary: 60H10, 49N25, 60G52	}

\begin{document}
	\begin{abstract}In this work, we investigate positive recurrent L\'evy diffusions driven by appropriately scaled  Brownian motion and $\alpha$--stable process (with $1<\alpha<2$) in the small noise regime. Supposing that in the vanishing noise limit, our L\'evy diffusion approaches a deterministic system  with a unique asymptotically stable fixed point, we show that the  limiting behavior of the one-dimensional marginal distribution at large times   is dictated by the  optimal value of a deterministic control problem, just as in the classical case of diffusions driven by small variance Brownian motion. In  our case, the control is allowed to have two parts:  continuous control and impulse control.
	\end{abstract}

        	\maketitle

	\section{Introduction}  
In the seminal work of M. I. Freidlin and A. D Wentzell in \cite{freidlin2012}, the authors investigated  small noise asymptotics of invariant measures of continuous diffusions. They consider an $\RR^p$--valued diffusion given by 
	$$ Z^n(t)= Z^n(0)+ \int_0^t b\big(Z^n(s)\big)\D s + a_n  W(t),$$
where, $a_n\downarrow 0$ as $n\to\infty$ and $b: \RR^p\rightarrow \RR^p$ be such that  the associated ordinary differential equation (ODE)  $\dot{Z}(t)= b\big(Z(t)\big)$  has a unique asymptotically stable equilibrium (say at $z_0$). Then  the  unique invariant measure $\tilde \pi^n$ of $Z^n$ satisfies a large deviation principle (LDP), namely
	 \begin{align}\label{eq-ldp-cont}
	-\inf_{z\in A^\circ}I(z)\leq \liminf_{n\to\infty}a_n^2 \log \tilde \pi^n(A)\leq \limsup_{n\to\infty}a_n^2 \log\tilde \pi^n(A)\leq  -\inf_{z\in \overline A}I(z),
	 \end{align}
 for a Borel set  $A\subset \RR^p$, with  $A^\circ$ being the interior of $A$, $\overline A$  the closure of $A$ and
	 \begin{align}\label{eq-rf-cont}
	 I(z)\doteq \frac{1}{2}\inf \int_0^\infty \|u(t)\|^2 \D t
	 \end{align}
	with the infimum being over all measurable and square-integrable $u(\cdot)$ such that the solution  $z(\cdot)$ to the ODE
$$\dot z(t)= -b\big(z(t)\big)-u(t)\,, z(0)=z$$
        satisfies  $z(t)\to z_0$ as $t\to\infty$. The small noise asymptotics of invariant measures is dictated by a rate function which can be expressed as the value function of a  certain optimal control problem involving only continuous controls as in~\eqref{eq-rf-cont}. Therefore it is natural to consider jump diffusions driven by  an independent pair  $(W,L^\alpha)$  of an   $\RR^p$--valued Brownian motion and an $\RR^p$--valued  $\alpha$--stable process with ($1<\alpha<2$).  The  associated small noise asymptotics for this would be  dictated by an optimal control problem that involves both continuous and impulse controls. 

However, one needs to account for the fact that the sample-path LDP of $\{b_nL^\alpha:n\in \NN\}$ does not hold (see \cite[Section 4.4]{rhee2019sample}) and only  a sample-path weak LDP (WLDP) of $\{b_nL^\alpha: n\in \NN\}$ holds (see \cite[Section 3.2]{rhee2019sample}). Namely,  for large $n$,
	\begin{align*} \PP(b_n L^\alpha\in U )\gtrsim \exp\Big(  \big(\log b_n\big)\inf_{x\in U} I_L(x)\Big), \text { for an open set $U$,} \\
	\PP(b_n L^\alpha\in K )\lesssim \exp\Big(  \big(\log b_n\big)\inf_{x\in K} I_L(x)\Big), \text { for a compact set $K$}, \end{align*}
	where $I_L(x) $ is a discrete valued function equal to the total number of jumps  in all directions for a piecewise constant $x$ and $I_L(x)=\infty$, otherwise.

In this paper, we consider the following family of $p$--dimensional L\'evy diffusions  parametrized by $(n,\gamma)$ with $n\in \NN$ and $\gamma>0$,  given as the solution to the following equation:
\begin{align}\label{alphastable}
X^{n,\gamma}(t)= X^{n,\gamma}(0)+\int_0^tb\big(X^{n,\gamma}(s)\big)\D s + \frac{1}{\sqrt{\log n}}W(t)+ \frac{1}{n^\gamma} L^\alpha(t)\,.
\end{align}
where $W$ and $L^\alpha$ are $p$--dimensional Brownian motion and
$p$--dimensional $\alpha$--stable process with $1<\alpha<2$,
respectively. We prove in Theorem~\ref{thm-main} that $X^{n,\gamma}(T)$
satisfies a WLDP as $n\rightarrow \infty$ followed by $T\rightarrow \infty$. The choice of $\frac{1}{\sqrt{\log n}}$ and
  $\frac{1}{n^\gamma}$ is dictated by the above WLDP for $\alpha$-stable
  processes and the sample path LDP of $\{\frac{1}{\sqrt{\log n}} W: n\in \NN\} $\cite{schilder1966some}.  Note that larger the $\gamma$  smaller is the effect of $L^\alpha$ on the dynamics of $X^{n,\gamma}$ and vice versa.  To prove this, we use techniques from control theory mainly, dynamic programming.\\

The remainder of the introduction is organised as follows. In Section
\ref{sec:bac} we provide a brief review of the literature, followed by statement of main results in Section \ref{sec:main}. We conclude this section with a discussion of the significance of our main results, the novelty in our proof technique, and the organisation of the paper in Section \ref{sec:disc}.

 \subsection{Background and Literature}       \label{sec:bac}

 Our work initially began with the objective of understanding the
 behaviour of invariant measures for \eqref{alphastable} in the small
 noise limit viz. as $n$ goes to infinity. There is a wide literature in this area for diffusions driven by a brownian motion. We provide a brief review.

 In \cite{borkar2013asymptotics}, the authors consider how the
long-term distribution of the invariant measure of a large system of
interacting Markov chains behaves as the number of components grows.
Under mean-field coupling they show that when the limiting
deterministic McKean-Vlasov system has a unique stable equilibrium, the
empirical distribution concentrates there. For a diffusion process
with small random perturbations around a system with a unique globally
attracting stable equilibrium, it was shown in
\cite{day1988localization} that the stationary density becomes
essentially localized. Its values inside a bounded region around the
equilibrium depend less and less on the behavior of the system
outside that region as the noise goes to zero, leading to asymptotic
independence of the equilibrium density from distant coefficients and
strengthening earlier asymptotic representations of the invariant
density.

In \cite{hu2017hypo} the authors analyze a class of hypoelliptic
multiscale Langevin diffusions, focusing on both their large deviation
behavior and the small mass asymptotics of their invariant
measures. They prove that the LDP for the multiscale
system are consistent with those of its overdamped counterpart, and
they provide asymptotic expansions for the invariant density and
related hypoelliptic Poisson equations as the mass parameter tends to
zero, offering insight into how multiscale structure and
hypoellipticity influence rare events and long-term statistics.

 The asymptotic behavior of stationary densities for a family of
 Markov processes depending on a small noise parameter was considered
 in \cite{mikami1988asymptotic}. Under appropriate smoothness and
 stability conditions, they derive formal asymptotic expansions of the
 invariant density in powers of the small parameter, showing that the
 stationary distribution concentrates near the stable equilibrium of
 the deterministic part of the dynamics and that corrections to this
 concentration can be systematically computed.

In the context of stochastic partial differential eequations, in \cite{cerrai2022large} the authors consider the two-dimensional incompressible Navier–Stokes equations on a torus driven by Gaussian noise that is white in time and has spatial correlation whose scale shrinks to zero together with the noise intensity. They prove an LDP both for the solutions and for the family of invariant measures as the noise strength and correlation length tend simultaneously to zero under an appropriate scaling, showing that the stationary distributions concentrate according to a rate function determined by the quasi-potential of the associated controlled deterministic Navier–Stokes dynamics.

As mentioned earlier,  we were not able to derive an LDP for the
invariant measure for \eqref{alphastable}. This is due to the fact that the sample path LDP of the  $\alpha$--stable process with ($1<\alpha<2$) does not hold under conventional space-time scaling, but a weak sample path LDP holds, \cite{rhee2019sample}. However, we  do consider  $n\rightarrow \infty$ followed by $T\rightarrow \infty$ limit. We state our main result precisely next.

\subsection {Main result}\label{sec:main} 

Let $(\Omega, \calF, \{\calF_t\}_{t \geq 0}, \PP)$ denote the underlying probability space on which \eqref{alphastable} is defined. Let $\|\cdot\|$ denote the Euclidean norm on $\RR^p$ with $\langle \cdot,\cdot\rangle$ denoting the corresponding inner product. Let  $\RR_+$ be the set of non-negative real numbers.  Let $\frD_T^p$ ($\frC_T^p$, resp.) be the set of $\RR^p$--valued  right continuous functions with left limits (continuous functions, resp.)  on $[0,T]$ equipped with the Skorohod topology (uniform topology, resp.). When  $T=\infty$, we simply write $\frD^p$ ($\frC^p$, resp.)  For a piecewise constant c\'adl\'ag function $\xi\in \frD^1$, $\D_+(\xi)$ (resp., $\D_-(\xi)$) denotes the number of upward (resp., downward) jumps. Also, for any path  $f\in \frD^p$, $f([a,b])$, for $0\leq a<b<\infty$, refers to the path from time $a$ to time $b$ with the end points $a,b$ included.

\begin{definition}\label{def-ldp} For a Polish space $(S,\calS)$, A family  $\{Z^n: n\in \NN\}\subset \calP(S)$ of $S$--valued random variables is said to satisfy a large deviation principle (LDP) with rate function $I$ and rate $c_n$,  if  $I:S\rightarrow [0,\infty]$ is a lower semi-continuous function  such that $I_l\doteq \{x\in S:I(x)\leq l\}$ is compact in $S$ for every $l>0$, and
	\begin{align}\label{def-ldp-c} \limsup_{n\to\infty} \frac{1}{c_n}\log \PP(Z^n\in C)\leq -\inf_{x\in C} I(x),\text{ for every closed set $C\subset S$},\end{align}
	\begin{align}\label{def-ldp-u} \liminf_{n\to\infty} \frac{1}{c_n}\log \PP(Z^n \in U) \geq -\inf_{x\in U}I(x), \text{ for every open set $U\subset S$.}  \end{align}
If $\{Z^n:n\in \NN\}$ is such that $I_l$ is not necessarily compact for some $l>0$, ~\eqref{def-ldp-c} is satisfied only when $C$ is compact, and~\eqref{def-ldp-u} holds for all open sets $U$, then we say that $\{Z^n:n\in \NN\}$ satisfies a weak large deviation principle (WLDP) with rate function $I$ and rate $c_n$.
\end{definition}

To keep the expressions concise, we define $$W^n(t)\doteq \frac{1}{\sqrt {\log n}} W(t)\text{ and }L^{n,\gamma}(t)\doteq \frac{1}{n^\gamma}L^\alpha(t), \text{ for $t\geq 0$}\,.$$   We now state the main assumption of the paper.
	 
\begin{assumption}\label{main-assump} The following conditions hold.
\begin{enumerate}
\item [(i)]
  The map $b:\RR^p\rightarrow \RR^p$ is a globally Lipschitz function and, $$X^{n,\gamma}(0),\, W,\,\text{and } L^\alpha \text{ are independent.}$$  Assume that  $\{X^{n,\gamma}(0):n\in \NN\}\subset \calP(\RR^p)$ is such that $X^{n,\gamma}(0)$ converges to $0$ in law and satisfies LDP with a strictly convex rate function $\widetilde{V}:\RR^p\rightarrow [0,\infty]$ and rate $\log(n)$, with minimum at $0$. 
    
\item [(ii)]The equation   $$ \frac{\D}{\D t}{\phi}(t,x)= b\big(\phi(t,x)\big),\, \phi(0,x)=x,$$
 has a unique globally asymptotically stable equilibrium at $x=0$ ($\doteq$ the zero vector of appropriate dimension). Here the choice of  $x=0$ is simply for convenience.
 	
	\end{enumerate}
\end{assumption}

Before we state the main result of the paper, we introduce a few important  definitions. 
\begin{definition} \label{def:imp} For $x,z\in \RR^p$ and $T>0$, we define a set of `controls' $\calS_x(z,T)$ as the set of all functions $(u,v): [0,T]\rightarrow \RR^p\times \RR^p$ such that the following hold:\\

\begin{enumerate}
	\item[(i)] $u:[0,T] \rightarrow \RR^p$ is Borel measurable and is such that $$\frac{1}{2}\int_0^T\|u(t)\|^2\D t<\infty\,.$$
	
	\item [(ii)] $v\in \frD^p_T$ is a piecewise constant map from $[0,T]$ to $\RR^p$ that satisfies the following property: if for $0\leq t\leq T$, $v(t)\neq v(t-)$, then there exists a unique $1\leq i\leq p$ such that $v_j(t)\neq v_j(t-)$, if and only if $i=j$.
	\end{enumerate}
	\end{definition}
	{ Let  $Y$ be an $\RR^p$--valued process that satisfies $Y(0)= z$, $Y(T)=x$ and  
	\begin{align}\label{ctdyn} \frac{\D }{\D t}{Y}(t)&= b(Y(t)) + u(t), \text{ for $t_{i-1}\leq t\leq t_i$,}\\
	 Y(t_i+)&= v(t_i)-v(t_i-), \label{ctdyn-i}
	\end{align}
	where, the set $\{t_i\}_{i=1}^k\subset [0,T]$ is such that  $v(t)= v(t-)$, for $t_{i-1}\leq t< t_i$ and $v(t_i)\neq v(t_i-)$. }

\medskip

When $Y(0)=z$ and $Y(T)\to x$ as $T\to \infty$ in the above definition, we denote the corresponding set of controls  by  $\calS_x(z)$  and   denote $\calS_x(0)$ by  $\calS_x$. We likewise define a `time-reversed' version $\calR_x(z,T)$ as above, except that~\eqref{ctdyn} and~\eqref{ctdyn-i} are replaced by:
\begin{align} \frac{\D }{\D t}{ Y}(t)&= -b( Y(t)) + u(t), \text{ for $t_{i-1}\leq t\leq t_i$} \label{ctdynR} \\
	 Y(t_i+)&= v(t_i)-v(t_i+),  \label{ctdynR1}
	\end{align}
such that $Y(0)=x$ and $Y(T)=z$. $\calR_x(z)$ is defined as the set of controls such that $Y(T)\to z$ as $T\to\infty$ and $ \calR_x\doteq \calR_x(0) $.
Let $I_L:\frD^p\rightarrow [0,\infty]$ be a discrete valued function defined as
\begin{align}\label{def-IL} I_L(v)\doteq \sum_{i=1}^p \widetilde I_L(v_i) \text{ with } v=(v_1,v_2,\ldots,v_p), \end{align}
where, $\widetilde I_L:\frD^1\rightarrow [0,\infty]$ is defined as 
\begin{align}\label{def-rf-L}
\widetilde I_L(\phi)\doteq \begin{cases}\alpha\big( \D_+ (\phi) +\D_-(\phi)\big), &\text{ whenever $\phi\in \frD^1$ is  a piecewise constant function.}\\
 \infty, &\text{ otherwise.}
\end{cases}
\end{align}

We first establish that the one-dimensional marginals of the process $X^{n,\gamma}$  satisfying~\eqref{alphastable} satisfies a WLDP.

\begin{proposition}\label{prop-finite}
Let $(u,v)$ and $\calS_x(z,T)$ as in Definition~\ref{def:imp} and $I_L$ be as defined in~\ref{def-IL}. Suppose Assumption~\ref{main-assump} holds. Then for every $T>0$,  $\{X^{n,\gamma}(T):n\in \NN\}$ satisfies WLDP with rate $\log(n)$ and rate function $V_{\gamma,T}:\RR^p\rightarrow \RR_+$   given by  
\begin{align}\label{eq:hjb-mod-rep}
V_{\gamma,T}(x)= \inf_{ \stackrel{(u,v) \in \calS_x(z,T)}{ z\in \RR^p}} \Big( \widetilde V(z)+\frac{1}{2}\int_0^T \|u(s)\|^2 \D s + \gamma I_L(v)\Big)\, .
\end{align} 

\end{proposition}

We then consider the $T\rightarrow \infty$ limit, leading to  the  main result of this paper stated below.\\

\begin{theorem}\label{thm-main} Let $(u,v)$ and $\calS_x(z,T)$ as in Definition~\ref{def:imp}, $I_L$ be as defined in~\ref{def-IL}. and $X^{n,\gamma}(\cdot)$ be as in \eqref{alphastable}. Under Assumption~\ref{main-assump}, the following holds: for every compact set $K\subset \RR^p$
\begin{align}\label{eq-ub}
\limsup_{T\to\infty} \limsup_{n\to\infty} \frac{1}{\log(n)} \log \PP\big(X^{n,\gamma}(T) \in K\big)\leq -\inf_{x\in K} V_\gamma(x)
\end{align}
and for every open set $U\subset \RR^p$,
\begin{align}\label{eq-lb}
\liminf_{T\to\infty} \liminf_{n\to\infty} \frac{1}{\log(n)} \log  \PP\big(X^{n,\gamma}(T) \in U\big)\geq -\inf_{x\in U} V_\gamma(x)\,,
\end{align}
with  $V_\gamma:\RR^p\rightarrow \RR_+$ given by
    \begin{align}\label{eq:V_exp}
	V_\gamma(x)\doteq  \inf_{ (u,v) \in \calR_x} \Big(  \frac{1}{2}\int_0^\infty \|u(s)\|^2 \D s + \gamma I_L(v)\Big).
	\end{align}
\end{theorem}

\medskip

\subsection{Discussion and Outline} \label{sec:disc}
 We begin with some remarks.       
  
    \begin{enumerate}
\item[(a)] Observe that $V_\gamma$ from~\eqref{eq:V_exp} is independent of $\widetilde{V}$. Due to the non-existence of LDP for $\{n^{-\gamma}: n\in \NN\}$, it is not clear if the limits in~\eqref{eq-ub} and~\eqref{eq-lb} are interchangeable.

\item[(b)] As one may notice, there is a  difference in the scaling $(\log n)^{-\frac{1}{2}}$  of Brownian motion $W$ and  the scaling $n^{-\gamma}$ of $L^\alpha$ in~\eqref{alphastable}.  This is because in this work, we are interested in obtaining the rate function $V_\gamma$ which includes the effect of both $W$ and $L^\alpha$. We will see from Proposition~\ref{prop-w-ldp} and Proposition~\ref{prop-L-wldp} that the scaling of $W$ and $L^\alpha$  chosen in~\eqref{alphastable} ensures comparable strengths of $W^n$  and $L^{n,\gamma}$ for large $n$. This ensures that  the limiting behavior is dictated by both $W$ and $L^\alpha$. 

\item[(c)] Equation \eqref{eq:V_exp}  characterizes $V_\gamma $ as the value function for the control problem that seeks to minimize the `total cost' given by the expression in parentheses.  Replacing the infinite time horizon in~\eqref{eq:V_exp} with a finite horizon, we can compare the resulting optimal control problem with the one introduced in \cite{bl75}. With inclusion of a discounted term $e^{-\lambda t}$ ($\lambda>0$) in both the cost terms in~\eqref{eq:V_exp}, we can compare the resulting control problem with the one analyzed in \cite{barles1985deterministic}. The classical `continuous' control is $\{u(t), t\geq 0\},$ and the impulse control is the jumps in $\{v(t): t \geq 0\}$. There is a key difference, however, viz., that for us the cost due to the jumps in $\{Y(t): t \geq 0\},$ depends only on their number and not on their sizes. Importantly, this means that a nearly optimal `control' pair $(u,v)$ comprises of only a finite number of impulses throughout the infinite time horizon.

\item [(d)] {Let  $ \underline \frB\doteq \{z: I(z)\leq  p \gamma \alpha\}$ with $I$ as defined in~\eqref{eq-rf-cont}.  The dependence of $V_\gamma$ on $\gamma$ is only through the term $\gamma I_L(\cdot)$ in~\eqref{eq:V_exp}. From  convexity and  compactness of level sets of $I$ from~\eqref{eq-rf-cont}, it is clear that $\underline\frB$ grows as $\gamma \uparrow \infty$  and it  shrinks to $\{0\}$ as $\gamma \downarrow 0$. This means that higher the $\gamma$, more severe is the penalty for implementing impulses.  Hence it is not optimal to implement an impulse unless $x$ is far enough from origin, so that a purely continuous control cannot be optimal.  On the other hand, as $\gamma \downarrow 0$, for any purely continuous control  to be optimal (or nearly optimal) it would have to accomplish the job of taking the dynamics to origin with cost no greater than $p \gamma \alpha$. This is because with cost $p \gamma  \alpha $ one can take the dynamics from any $x\in \RR^p$ to the origin using a series of $p$ impulses.  This implies that larger the $\gamma$, smaller is the effect of $L^\alpha$ and vice-versa. }

  \item[(e)] An interesting variant of the model considered in~\eqref{alphastable} is the following: instead of considering the $n^{-\gamma}L^\alpha$, we can consider $\widehat \gamma=(\gamma_1,\gamma_2,\ldots,\gamma_p)$ with $\gamma_i>0$ for $i=1,\ldots,p$ and 
$$\widehat L^{n,\widehat \gamma}=\big(n^{-\gamma_1}L^\alpha_1,\ldots, n^{-\gamma_p}L^{\alpha}_p\big)\,.$$

Following the arguments in the proof of Theorem~\ref{thm-main}, we can derive results analogous to~\eqref{eq-ub} and~\eqref{eq-lb} with 
\begin{align*}
	\widehat{V}_\gamma(x)\doteq  \inf_{ (u,v) \in \calR_x} \Big(  \frac{1}{2}\int_0^\infty \|u(s)\|^2 \D s + \widehat  I_L(v)\Big).
	\end{align*}
	instead of $V_\gamma$, where, 
	$$ \widehat I_L(v)\doteq  \sum_{i=1}^p \gamma_i\widetilde I_L(v_i) \text{ with } v=(v_1,v_2,\ldots,v_p)\,.$$
In this case, an impulse is penalized depending on the component along which it has been implemented and can be handled as above.

  \end{enumerate}

\medskip

\paragraph{\bf Sketch of the proof}\label{sec-sketch} 
We now briefly discuss our approach to proving~ Theorem~\ref{thm-main}. It rests on two key steps: \\

\noindent 1.\ \textit{A ``contraction" type principle:} It is well-known (and is referred to as contraction principle; see \cite[Theorem 4.2.1]{dembo2009large}) that if a family of random variables $\{Z^n:n\in \NN\}$ satisfies an LDP and $F$ is a continuous function with appropriate domain and range, then $\{F(Z^n): n\in \NN\}$ also satisfies an LDP.   In our case, the family of random variables of relevance is $$\big\{(X^{n,\gamma}(0), \frac{1}{\sqrt {\log n}} W, \frac{1}{n^\gamma} L^\alpha): n\in \NN\big\}$$ 
as this is the only source of randomness and the map from $(X^{n,\gamma}(0), a_nW, b_n L^\alpha)$, for $T>0$ to $X^{n,\gamma}(T)$ is continuous in the appropriate topology; see Lemma~\ref{lem-cont-map}. However, as $\{\frac{1}{n^\gamma} L^\alpha: n\in \NN\}$ does not satisfy an LDP (but only  WLDP), an immediate application of contraction principle is not possible in order to conclude that $X^{n,\gamma}(T) $ satisfies an LDP (or even WLDP). Using the pure-jump nature of $L^\alpha$, we  show that  WLDP holds for $X^{n,\gamma}(T)$; see Lemma~\ref{th:weakldp-contprinc}. \\

\noindent 2.\ \textit{Stability of a  Bellman type equation:}  Applying a `contraction type' principle along with Markovian property of the process $X^{n,\gamma}$, we express the rate function of $\{X^{n,\gamma}(T): n\in \NN\}$ in terms of the rate functions of $X^{n,\gamma}(0)$, ${\log n}^{-\frac{1}{2}} W$ and $n^{-\gamma}L^\alpha$. This results in a Bellman type equation (depending on $T$) ; see \eqref{eq:hjb-mod-rep}. More precisely, the rate function associated with WLDP of  $\{X^{n,\gamma}(T): n\in \NN\}$ is expressed as the optimal cost of a control problem with initial cost. With a time reversal, we recast the control problem as an optimal control problem with a terminal cost and the rate function associated with WLDP of  $\{X^{n,\gamma}(T): n\in \NN\}$ as the value function of this optimal control. The advantage of this recasting is that we can now take $T\rightarrow \infty$ in the dynamic programming equation.  \\

\paragraph{\bf Layout of the paper} The rest of the paper is organised as follows. In Section~\ref{s:auxiliary}, we prove Proposition~\ref{prop-finite} with the help of auxillary lemmas required to prove the main result. Section~\ref{s:proof} contains the proof of Theorem~\ref{thm-main}. Finally, in the appendix, we state and prove a variant of an  elementary result from deterministic control theory.

\section{A ``contraction" type principle}\label{s:auxiliary} 
The content of this section is to establish the WLDP for $\{X^{n,\gamma}(T): n\in \NN\}$ for $T>0$ whenever  $\{X^{n,\gamma}(0):n\in \NN\}$ satisfies the hypotheses of Theorem~\ref{thm-main}.   To begin with, for every $T>0$, let us define a map $F_T:\RR^p\times \frC^p_T\times \frD^p_T\rightarrow \frD_T^p$ that maps $ (X^{n,\gamma}(0), W^n,L^{n,\gamma})$ to $X^{n,\gamma}([0,T])$ \emph{i.e.,}
 \begin{align}\label{def-F} F_T: (x, w_1,w_2)\mapsto y\in \frD^p_T,\end{align} where $y$ is given as a solution to the following equation: for $0\leq t\leq T$,
 $$ y(t)=x+\int_0^t b(y(s))\D s + w_1(t)+w_2(t)\,.$$
\begin{lemma}\label{lem-cont-map}
  For every $T>0$, the map $F_T$ is continuous on $\RR^p\times \frC^p_T\times \frD^p_T$.
\end{lemma}

\begin{proof} The lemma follows from a multi-dimensional version of \cite[Lemma 3.1]{wei2021large}. 
\end{proof}

\medskip

It is a well-known classical result that for $T>0$, the $\frC^p_T $--valued family of random variables $\{W^n([0,T]): n\in\NN\}$ satisfies a sample-path LDP. Observe that if $\{L^{n,\gamma}([0,T]):n\in \NN\}$ also satisfies a sample-path LDP, then a simple application of contraction principle (see \cite[Theorem 4.2.1]{dembo2009large}) using  Lemma~\ref{lem-cont-map} and the fact that $\{X^{n,\gamma}(0):n\in \NN\}$ satisfies the hypothesis of Theorem~\ref{thm-main} concludes the LDP of $\{X^{n,\gamma}(T):n\in \NN\}$.  Unfortunately,   from a recent result in    \cite{rhee2019sample} (see Section 4.4 of that paper), it is known that a sample-path LDP does not hold for $\{L^{n,\gamma}([0,T]): n\in \NN\}$. However, Theorem 4.2 of  \cite{rhee2019sample} concludes that  the $\frD_T^p$--valued family of  random variables $\{L^{n,\gamma}([0,T]): n\in \NN\}$ satisfies a sample-path WLDP.   Therefore, one can ask whether the $\{X^{n,\gamma}(T):n\in \NN\}$ inherits the LDP (or at least WLDP). Addressing this question is the content of the rest of this section.  Unlike in the case of LDP, there is no analog of a contraction principle in case of WLDP.   However, using the fact that $\{L^{n,\gamma}: n\in \NN\}$ is a pure jump process, we prove a weak version of the contraction principle that suffices for our case. \\

We begin by stating two existing results from the literature. The first is the well-known result on sample-path LDP of $\{W^n([0,T]): n\in \NN\}$  (see \cite[Theorem 5.2.3]{dembo2009large}), and second is a relatively new result on the sample-path WLDP of $\{L^{n,\gamma}([0,T]): n\in \NN\}$.
\begin{proposition}\label{prop-w-ldp}
For every $T>0$, the family of $\frC_T^p$--valued random variables $\{W^n([0,T]): n\in \NN\}$ satisfies LDP with rate $\log(n)$ and rate function $I_W:\frC^p_T\rightarrow [0,\infty]$ given by
\begin{align}
  I_W(\xi)=\begin{cases} \frac{1}{2}\int_0^T\|u(s)\|^2\D s, & \text{whenever $\xi$  is absolutely continuous}\\
  &\text{ and $\xi(\cdot)\doteq \int_0^\cdot u(s)\D s$ on $[0,T]$,}\\
\infty, &\text{otherwise.}
\end{cases}
\end{align}
\end{proposition}
In the above, we suppress the dependence of the rate function $I_W$ on $T$, as it is evident from the context.\\

The following is the multidimensional version of the result  proved in \cite{rhee2019sample}. The authors of  \cite{rhee2019sample} study a heavy tail version of a large deviation principle for $\alpha$--stable process $L^{n,\gamma}$ with $1<\alpha<2$, using a scaled process $$\widetilde L^{n,\gamma}(t)\doteq \frac{1}{n^{\frac{\gamma \alpha}{\alpha-1}}}L^\alpha\left(n^{\frac{\gamma \alpha}{\alpha-1}}t\right), \text { for $t\geq 0$}\,.$$

\medskip

\begin{proposition}\label{prop-L-wldp}
For every $T>0$, the family of $\frD^p_T$--valued random variables $\{ \widetilde L^{n,\gamma}([0,T]): n\in \NN\}$ satisfy WLDP with rate $\log(n)$ and rate function $ \gamma I_L$ with $I_L:\frD^p_T \to \RR$ obtained by restricting $I_L$ defined in~\eqref{def-IL} to $[0,T]$.
\end{proposition}
\begin{proof} The proof of the result for $p=1$ follows directly from \cite[Theorem 4.2]{rhee2019sample}. Since $\widetilde  L^{n,\gamma}_i$  are i.i.d.\  for $1\leq i \leq p$, a direct application of \cite[Lemma 2.8]{lynch1987large} extends the result to $p>1$. 
\end{proof}

\medskip

From the self-similarity of process $L^\alpha$, we know that the laws of $\{L^\alpha(nt): t\in [0,T]\}$  and $\{n^{\frac{1}{\alpha}} L^\alpha(t): t\in[0,nT]\}$ are identical for every $T>0$. This gives us the  following immediate corollary.
\begin{corollary}\label{cor-L}
For every $T>0$, the family of $\frD^p_T$--valued random variables $\{ L^{n,\gamma}([0,T]): n\in \NN \}$ satisfy WLDP with rate $\log(n)$ and rate function $\gamma I_L$.
\end{corollary}

\begin{remark}
In this paper, we consider $L^{n,\gamma}$, instead of $\widetilde L^{n,\gamma}$.  This brings us to the reason behind choosing $1<\alpha<2$. The reason is two-fold. Firstly, recall that we are interested in studying the small noise asymptotics \emph{i.e.,} asymptotics as the noise goes to zero or equivalently, as $n\to\infty$. This means that whenever $1<\alpha<2$ and $n\to\infty$, $n^{\frac{\gamma(\alpha-1)}{\alpha}}\to\infty$ and limiting behaviour of $L^{n,\gamma} $ is the same as that of law of a large numbers type scaling limit (which in this case is equivalent to $\widetilde L^{n,\gamma}$). This also helps us to use the results from \cite{rhee2019sample}. Observe that when $0<\alpha\leq 1$, this is no longer the case. To the best of our knowledge there are no results in the literature on  the limiting behaviour of  $\{\widetilde L^{n,\gamma}:n\in\NN\}$ whenever $0<\alpha\leq 1$.
\end{remark}

 We  need the following result that is a part of \cite[Theorem 3.4]{rhee2019sample} which is for  $p=1$ case. For $j,k\in \Z_+$, let $\calD_{j,k}\subset \frD^1_T$ be the set of all $\xi\in \frD^1_T$ such that $\xi$ is piecewise constant with exactly $j$ upwards and $k$ downwards jumps. \\
 
\begin{proposition}\label{prop-heeavyldp}
	Suppose $A\subset \frD^1_T$ is a measurable set and $A\cap\calD_{j,k}=\emptyset$ for every $j,k\in \Z_+$.  If in addition, $A$ is bounded away from the set
	$$\{\xi\in\frD^1_T:  I_L(\xi)\leq \alpha (l+m)\},$$
	for some $l,m\in \Z_+\setminus\{0\}$. 
	Then
	$$
	\lim_{n\to \infty} \frac{\PP(L^{n,\gamma}\in A)}{n^{-\alpha(l+m)}}=0.$$
\end{proposition}

\medskip

Let $B_R$ denote the closed $R$-ball in $\frD^p_T$ centered at the origin under the supremum norm \emph{i.e.,} $B_R\doteq \{x\in \frD^p_T: \sup_{0\leq t\leq T}\|x(t)\|\leq R\}$. Let $J_\delta$ denote the set of all piecewise constant c\'adl\'ag functions with jumps of  size at least $\delta$.
\begin{lemma}\label{levelset-closedball-compact} The set $\{x\in \frD^p_T :  I_L(x)= l\}\cap B_R\cap J_\delta$ is compact in $\frD^p_T$ for every $l,R,\delta> 0$. \\

\end{lemma}
\begin{proof} The result follows from a simple verification of  conditions of \cite[Theorem 14.4]{billingsley2013convergence} for every $l,R,\delta>0$.
\end{proof}

\medskip

\begin{corollary}\label{cor:weakldp-infjumps}For $p=1$, the following holds:
	\begin{align*}
	\limsup_{n\to \infty} \frac{1}{\log(n)} {\log \PP\left( L^{n,\gamma} \in \big(\{ I_L=\infty\}\cup_{k\in \NN}\{ I_L\leq k\}\big)^c\right)} =-\infty
	\end{align*}
\end{corollary}

\begin{proof}
	Since $L^{n,\gamma}$ is a pure jump process, the  probability $\PP\circ(L^{n,\gamma})^{-1}$ is supported only on  the set of piecewise constant functions in $\frD^p_T$ with finite or infinitely many jumps. In other words, the
	support of  $\PP\circ(L^{n,\gamma})^{-1}$ is 
	$$ \{ I_L=\infty\}\cup_{k\in \NN} \{ I_L\leq k\}.$$
This immediately gives us the result.
\end{proof}

\medskip

\begin{remark} \label{rem-multi}For $p>1$, since $L^{n,\gamma}_i$ and $L^{n,\gamma}_j$, for $1\leq i< j\leq p$ are i.i.d., we can also conclude the multidimensional versions of Proposition~\ref{prop-heeavyldp} and Corollary~\ref{cor:weakldp-infjumps}.
\end{remark}

\medskip

We now proceed to prove the main result of this section.
Suppose $\{\beta^n:n\in \NN\}$ is a family of $\RR^p$--valued random variables. For $T>0$, define a $\calY\doteq \RR^p\times \frC_T^p\times \frD^p_T$--valued random variable
$$ Y^n \doteq \left(\beta^n, W^n, L^{n,\gamma}\right).$$
Here $\beta^n, W^n $ and $L^{n,\gamma}$ are independent for every $n$. Recall that $\gamma>0$.
  Let $\calZ$ be a Polish space and $G:\calY\rightarrow \calZ$ be a continuous map.  Also, let $Z^n=G(Y^n)$ be a $\calZ$--valued random variable for every $n$. 
\begin{lemma}\label{th:weakldp-contprinc} Suppose the family $\{\beta^n:n\in \NN\}$ of $\mathbb{R}^p$--valued random variables  satisfies LDP with rate $\log(n)$ and rate function $I_\beta$. Then, $\{Z^n:n\in \NN\}$ satisfies  WLDP with rate $\log(n)$ and rate function $I_Z:\calZ\rightarrow [0,\infty]$ given by 
	$$ I_Z(z)\doteq \inf_{y\in \calY:z  =G(y)}I_Y(y),$$
	where
	$$ I_Y(y)\doteq I_Y(y_1,y_2,y_3)\doteq I_\beta(y_1)+ I_W(y_2)+ \gamma I_L(y_3)\,.$$
	
\end{lemma}

\medskip

\begin{remark} In the proof of the above lemma,  the structure of $\R^p$ does not play a role.  The result holds even if  $\R^p$ is replaced by a Polish space.
\end{remark}

\medskip

\begin{proof} From \cite[Lemma 2.8]{lynch1987large}, $\{Y^n: n\in \NN\}$ satisfies WLDP with rate $\log(n)$ and rate function $I_Y$.
For any open set $U\subset \calZ$, $G^{-1} U\subset \calY$ is also open, due to continuity of $G$. From the definition of WLDP, we have
\begin{align*} \liminf_{n\to\infty} \frac{1}{\log(n)} {\log \PP(Z^n \in U)} & = \liminf_{n\to\infty} \frac{1}{\log(n)}{\log \PP(Y^n \in G^{-1}U)}\\
  &\geq -\inf_{y\in G^{-1}U}I_Y(y)\\
  &= -\inf_{z\in U}I_Z(z)\,. \end{align*}
The second equality above follows from the definition of $I_Z$. It only remains to show that for every compact set $K\subset \calZ$, 
	\begin{align}\label{eq-des-ub} \limsup_{n\to\infty} \frac{1}{\log(n)}\log\PP(Z^n\in K) \leq -\inf_{z\in K} I_Z(z)\,.\end{align}
	Fix a compact set $K\subset \calZ$ and define $ \calK\doteq G^{-1}K$ which is not necessarily compact.
To keep the expressions concise, we introduce some additional notation. For $M>0$, define
	\begin{align*}
	\calA_{M}&\doteq  \{y_1\in \RR^p: I_\beta(y_1)\leq M\}\times \{y_2\in \frC_T^p: I_W(y_2)\leq M\}\times \calG_{M},
	\end{align*}
where $\calG_{M}=\{x\in \frD^p_T :  I_L(x)= M\}\cap B_M\cap J_{M^{-1}}$. That is $\calG_M$ is the set of piecewise constant functions in $\frD^p_T$ that are confined to the closed ball of radius $M$ centred at the origin, with exactly $M$ jumps of size at least $M^{-1}$. From Lemma~\ref{levelset-closedball-compact}, $\calG_M$ is compact in $\frD^p_T$. Furthermore, using the compactness of level sets of $I_\beta$ and $I_W$, we obtain the compactness of $\calA_M$ in $\calY$. Now consider
	\begin{align*}
	\limsup_{n\to\infty} \frac{1}{\log(n)}{\log\PP\left(Z^n \in K\right)}= \limsup_{n\to\infty} \frac{1}{\log(n)}{\log\PP\left(Y^n \in \calK\right)}.
	\end{align*}
	We bound
	$ \limsup_{n\to\infty} \frac{1}{\log(n)} {\log \PP\left(Y^n \in \calK\cap \calA_{M}\right)}$ and
	$ \limsup_{n\to\infty} \frac{1}{\log(n)} {\log \PP\left(Y^n \in \calK\cap \calA_{M}^c\right)}$
	separately.
	 Using WLDP of $\{Y^n:n\in \NN\}$ and the compactness of $\calK\cap \calA_M$ (note that $\calK$ is a closed set), we have
	\begin{align}
	\limsup_{n\to\infty} \frac{1}{\log(n)} {\log \PP\left(Y^n \in \calK\cap \calA_{M}\right)}\leq -\inf_{y\in \calK\cap \calA_M} I_Y(y)\leq -\inf_{y\in \calK} I_Y(y)= -\inf_{z\in K } I_Z(z)\,.
	\end{align}
	Again, to get the equality we use the definition of $I_Z$. For $k\in \NN$, let  $$\calH_k\doteq  \{y_1\in \RR^p: I_\beta(y_1)>M\}\times \{y_2\in \frC_T^p: I_W(y_2)>M\}\times \calG_{M+k} \,.$$
	Then we have for every $k\in\NN$,
	\begin{align*}
	\limsup_{n\to\infty} \frac{1}{\log(n)} {\log \PP\left(Y^n \in {\calH_k}\right)}&\le \limsup_{n\to\infty} \frac{1}{\log(n)} {\log \PP\left(\beta^n: I_\beta(\beta^n)>M\right)}\\
	&\quad+\limsup_{n\to\infty} \frac{1}{\log(n)} {\log \PP\left(W^n: I_W(W^n)>M \right)}\\
	&\quad+\limsup_{n\to\infty} \frac{1}{\log(n)} {\log \PP\left(L^{n,\gamma}\in \calG_{M+k} \right)} \\
	&\leq \limsup_{n\to\infty} \frac{1}{\log(n)} {\log \PP\left(L^{n,\gamma}\in \calG_{M+k} \right)} \,.
	\end{align*}
	In the above, we use the fact that $\beta^n$, $W^n$ and $L^{n,\gamma}$ are independent for every $n\in\NN$ and $\gamma>0$.
	To get the last line, we use  the LDP of $\{(\beta^n,W^n): n\in \NN\}$ to infer that  
	$$\limsup_{n\to\infty} \frac{1}{\log(n)} {\log \PP\Big(\beta^n: I_\beta(\beta^n)>M\Big)}\leq 0, \  \limsup_{n\to\infty} \frac{1}{\log(n)} {\log \PP\Big(W^n: I_W(W^n)>M \Big)}\leq 0\,.$$
	Since $\calG_{M+k}$ is compact in $\frD^p_T$, we further get
	\begin{align}\label{eq-h-k}
	\limsup_{n\to\infty} \frac{1}{\log(n)} {\log \PP\left(Y^n \in {\calH_k}\right)}\leq -\inf_{y\in \calG_{M+k}}  I_L(y)= -M-k\,.
	\end{align}
		Also, for $\calH_\infty\doteq \RR^p \times \frC^p_T\times \{y\in \frD^p_T: I_L(y)=\infty\},$ using Proposition~\ref{prop-heeavyldp} along with Remark~\ref{rem-multi}, 	  we have
	\begin{align}\label{eq-h-infty}
	\limsup_{n\to\infty} \frac{1}{\log(n)} {\log \PP\left(Y^n \in {\calH_\infty}\right)}=-\infty\,. 
	\end{align} 
	From the above definitions of $\calH_k$ and $\calH_\infty$, it is clear that 
		\begin{align}\label{eq-dis-union}
	\calA_{M}^c\subset \big(\RR^p \times \frC^p_T\times Q\big)\cup \calH_\infty\cup\big( \cup_{k=1}^\infty {\calH}_{k}\big)\,.
	\end{align} 	
Here, $Q$ is the set of all functions $\xi$ such that $\xi_i$ is not piecewise constant for at least one $1\leq i\leq p$.  From Corollary~\ref{cor:weakldp-infjumps} and the fact that $L_i^{n,\gamma}$ and $L_j^{n,\gamma}$ are pairwise independent for $1\leq i\neq j\leq p$, we have 
\begin{align}\label{eq-supp}
\limsup_{n\to\infty} \frac{1}{\log(n)} {\log \PP\Big(Y^n \in \big(\RR^p \times \frC^p_T\times Q\big)\Big)}=-\infty\,.
\end{align}
	Therefore, 
	\begin{align*}
	\limsup_{n\to\infty}& \frac{1}{\log(n)} {\log \PP\Big(Y^n \in (\calK \cap\calA^c_{M})\Big)}\\
	&\leq \limsup_{n\to\infty} \frac{1}{\log(n)} {\log \PP\Big(Y^n \in \calA^c_{M}\Big)}\\
	&\leq 
	\max\bigg\{\sup_{k\in \NN}\Big\{ \limsup_{n\to\infty} \frac{1}{\log(n)} {\log \PP\Big(Y^n \in  \calH_{k}\Big)}\Big\}, \limsup_{n\to\infty} \frac{1}{\log(n)} {\log \PP\Big(Y^n \in Q\Big)},\\
	&\qquad \limsup_{n\to\infty} \frac{1}{\log(n)} {\log \PP\Big(Y^n \in \calH_{\infty}\Big)}\bigg\}\\
	&\leq -M\,.
	\end{align*}
	In the above, the second inequality follows from~\eqref{eq-dis-union} and the third inequality follows from~\eqref{eq-h-k},~\eqref{eq-h-infty} and~\eqref{eq-supp}. To summarize, we have
	\begin{align*}
	\limsup_{n\to\infty} \frac{1}{\log(n)} {\log \PP\left(Z^n \in K\right)}=\limsup_{n \to\infty} \frac{1}{\log(n)} {\log \PP\left(Y^n \in \calK\right)}\leq \max\left\{-\inf_{z\in K} I_Z(z),-M \right\}.
\end{align*}
Finally, taking $M\uparrow \infty$ proves the result.	
\end{proof}

\begin{proof}[Proof of Proposition~\ref{prop-finite}]From Lemma~\ref{lem-cont-map} and Lemma~\ref{th:weakldp-contprinc} it is immediate that  for $T>0$, that $\{X^{n,\gamma}(T):n\in\NN\}$ satisfies a WLDP with rate $\log(n)$  and rate function  $V_{\gamma,T}$ given by~\eqref{eq:hjb-mod-rep}.
\end{proof}

\section{Proof of Theorem~\ref{thm-main}}\label{s:proof}

We will now prove that $V_{\gamma,T}$ converges as $T\to\infty$ and identify the limit.  We achieve this by first reversing the time variable $t \to -t$, and use dynamics prescribed in \eqref{ctdynR}. We may rewrite \eqref{eq:hjb-mod-rep} therefore as
\begin{align}\label{eq-hjb}
V_{\gamma,T}(x)= \inf_{ \stackrel{(u,v) \in \calR_x(z,T)} {z\in \RR^p}} \Big( \widetilde V(z) +\frac{1}{2}\int_0^T \|u(s)\|^2 \D s +\gamma I_L(v)\Big).
\end{align} 

Our first result shows that we have uniform bounded and  pointwise convergence.

\begin{proposition}\label{prop-Vexp} The following hold:
\begin{enumerate}
\item [(i)] \begin{equation} \label{eq:ubdd} \sup_{T>0}\sup_{x\in \RR^p} V_{\gamma,T}(x)\leq p\gamma \alpha\,.\end{equation}
\item [(ii)] For $\gamma >0$,  $V_{\gamma,T} \rightarrow V_\gamma$   pointwise, where  $V_\gamma$ given by~\eqref{eq:V_exp}. 
\end{enumerate}
\end{proposition}

\begin{proof}  {  Fix $\epsilon>0$ and $x\in \RR^p$. Consider the trajectory defined as follows: $Y(0)=x$ and $Y(T)=z$
\begin{align*} \frac{\D }{\D t}{ Y}(t)&=\begin{cases} -b( Y(t)) , &\text{ for $0\leq t\leq \epsilon $}\\
 0, &\text{ for $t>\epsilon$}
\end{cases}\\
 Y(\epsilon+)&= 0\,.
	\end{align*}
	Note that this trajectory corresponds to triplet $(z,u,v) \equiv (0, 0, v^*)$ with 
	\begin{align*}
	v^*(t) =\begin{cases} 0, &\text{ for $0\leq t<\epsilon$,}\\
	-\phi (\epsilon,x), &\text{ for $t\geq \epsilon$.} 
	\end{cases}
	\end{align*}
Therefore,  triplet $(z,u,v) = (0,0,v^*)$ is suboptimal for $V_{\gamma,T}$ and the associated cost for this triplet is bounded from above $p\gamma \alpha$, as $v^*$ can have at most $p$ non-zero components. This gives us $ V_{\gamma,T}(x)\leq p\gamma \alpha $ and proves part (i). }\\

The proof of part (ii) is divided into three steps.

\medskip

\noindent {\bf (Step 1) For $T>0$, there exists a  nearly optimal $(z,u,v)$ triplet such that it has  {at most $p$ one-dimensional  impulses}: } \\

In the following, we show that  whenever $V_{\gamma,T}(x)<\infty$, there exists a triplet $(z^{*,\delta},u^{*,\delta},v^{*,\delta})$ that is $\delta$-optimal for  the infimum~\eqref{eq-hjb}, such that {$I_L(v^{*,\delta})\in \{\alpha,\ldots, p\alpha\}$}. To that end, let $(z^\delta,u^\delta,v^\delta)$ be a $\delta$--optimal solution for the infimum in~\eqref{eq-hjb}. This implies that 
	 \begin{equation*}
	 V_{\gamma,T}(x)\geq \widetilde V(z^\delta )+ I_W(u^\delta) + \gamma I_L(v^\delta)-\delta \,. \label{deltabound}
	 \end{equation*}
	 Since $0\leq V_{\gamma,T}(x)<\infty$, we have $I_W(u^\delta)+ \gamma I_L(v^\delta)<\infty$ for every $T>0$.	 In particular, $K\doteq \gamma I_L(v^\delta)<\infty $ (assume {$K>p \gamma \alpha$}, else we are done: recall the definition of $I_L$ from~\eqref{def-IL}; also, note that $\frac{K}{\gamma \alpha}$ is an integer). This means that there are two  $\bar K\doteq \frac{K}{\gamma\alpha}$--dimensional vectors $\{t_j\}_{j=1}^{\bar K}$ and $\{\theta_j\}_{j=1}^{\bar K}$ such that  $v^{\delta}$ has jumps at instants $t= t_j$ with jump size $\theta_j$. 
Furthermore, there is a path $Y(\cdot)$ given by 
	    $Y(0)= x$. 
	    \begin{align*} \dot{Y}(t)&= -b(Y(t)) + u^\delta(t), \text{ for $t_{i-1}\leq t\leq t_i$}\\
	    Y(t_i+) &= v^\delta(t_i)-v^\delta(t_i-) = \theta_i, \ Y(T)= z^\delta\,.
	    \end{align*}
	    Recall that $v^\delta$ is a piecewise constant $p$-dimensional vector because $\bar K<\infty$. Define, $t^*= \min_{1\leq j
	    \leq \bar K}t_j$. We  construct $(z^{*,\delta},u^{*,\delta},v^{*,\delta})$ that is also $\delta$--optimal such that  {$I_L(v^{*,\delta})\in \{\alpha,\ldots,p\alpha\}$}. Let $Y^*(\cdot)$ denote the corresponding $Y(\cdot)$.  Define 
	 \begin{align*}
	 Y^*(0)&\doteq x,\, Y^{*}(T)=z^{*,\delta}\doteq z^\delta\\
	u^{*,\delta}(t)& \doteq \begin{cases}0, &\text{for  $0\leq t\leq t^*$}\,.\\	
u^{\delta}(t), &\text{for $t^*< t\leq T$\,.}
	\end{cases}
	 \end{align*}
	    Define $v^{*,\delta}$ by: 
	    \begin{align*}v^{*,\delta}(t)\doteq \begin{cases} 0, & \text{for $0\leq t\leq t^*$,}\\
	    Y^*(t^*)-Y^*(t^*+), &\text{for $t^*< t\leq T$.}
	    \end{cases}
	 \end{align*}
	 Noting that $z^{*,\delta}=z^\delta$, we now compare the associated costs of $(z^\delta, u^\delta, v^\delta)$ and $(z^{*,\delta}, u^{*,\delta}, v^{*,\delta})$. Thus
	 \begin{align*}
	 &\widetilde V(z^\delta) +\frac{1}{2}\int_0^T \|u^\delta(s)\|^2 \D s + \gamma I_L(v^\delta)- \widetilde V(z^{*,\delta}) -\frac{1}{2}\int_0^T \|u^{*,\delta}(s)\|^2 \D s - \gamma I_L(v^{*,\delta})\\
	 &\qquad= \frac{1}{2}\int_{0}^{t^*} \|u^\delta(s)\|^2 \D s + K  - {p\gamma \alpha} \ \text{ by the choice of $(z^{*,\delta},u^{*,\delta}, v^{*, \delta})$}\\
	 &\qquad > \frac{1}{2}\int_{0}^{t^*} \|u^\delta(s)\|^2 \D s, \text{ since {$K>p \gamma \alpha$}.}
	 \end{align*}
{Note that even though we used the impulse once, the cost incurred due to this can be greater than $\gamma\alpha$ (and will be at most $p\gamma\alpha$). This is because, the cost penalizes one-dimensional impulses and there are at most $p$ of them.} Therefore $\delta$--optimality of $(z^\delta, u^\delta, v^\delta)$ implies $\delta$--optimality of $(z^{*,\delta}, u^{*,\delta}, v^{*,\delta})$.	

\medskip

\noindent { \bf (Step 2) Constructing optimal controls as $T\to\infty$ in~\eqref{eq-hjb}:} \\

The proof follows along the same lines as that of \cite[Lemma 6]{biswas2009small}. The idea is to construct a function  $(u,v)$ on $\RR^+$ such that  $(u_t,v_t)_{\{t\in [0,T]\}}$ is the optimal control for every $T>0$. Consider the topology on $L^2_{\text{loc}}(\RR^+,\RR^p)$ given by the coarsest topology that makes the restriction $u\rightarrow (u_t)_{\{t \in [0,T]\}}\in L^2([0,T],\RR^p)$ (equipped with weak topology) continuous. Denote this space by  $L^2_{w,\text{loc}}(\RR^+,\RR^p)$.  Function spaces $L^2_{w}([0,T],\RR^p)$ and $L^2_w(\RR^+,\RR^p)$ are defined likewise. Then 
$$ (z,u,v)\mapsto \widetilde V(z) +\frac{1}{2}\int_0^T \|u(s)\|^2 \D s + \gamma I_L(v)$$ 
is a lower semi-continuous function on $\RR^p\times L^2_{w}([0,T],\RR^p)\times \frD^p_T$.  Then from~\eqref{eq-hjb}, it suffices to focus on $u$ such that 
$$\left\{u\in L^2([0,T],\RR^p): \int_0^T \|u(s)\|^2\D s\leq \sup_{T>0}2V_{\gamma,T}(x)\right\}.$$ 
By the Banach-Alaoglu theorem, the above set   is pre-compact in $L^2_w([0,T];\RR^p)$.  {From Step 1, we have already seen that it suffices to focus only on $v$ that have at most $p$ one-dimensional jumps and from the hypothesis that $\sup_{T>0}V_{\gamma,T}(x)<\infty$. We can further focus on $v$ that are  uniformly bounded.}  From this, we can conclude that nearly optimal controls $(u,v)$ can be restricted to a compact set (call it $\calU_T$) of $L^2_w([0,T],\RR^p)\times \frD_T^p$. From the compactness of level sets of $\widetilde V$ (as it is a rate function), we can restrict the nearly optimal $z$ to a compact set of $\RR^p$.  Therefore $$(z,u,v)\mapsto \widetilde V(z) +\frac{1}{2}\int_0^T \|u(s)\|^2 \D s +\gamma I_L(v)$$
attains a minimum value.  For $T'>T$, it is clear that for any $(u,v)\in \calU_{T'}$, we have $(u,v)_{\{t \in [0,T]\}}\in \calU_T$. Let $\calU^*$ be set of $(u,v)$ in $L^2_w(\RR^+,\RR^p)\times \frD^p$ such that $v$ is piecewise constant with at most $p$ one-dimensional  jumps, with
$$ \int_0^\infty \|u(s)\|^2\D s\leq \sup_{T>0}2V_{\gamma,T}(x) \text{ and } \sup_{t\geq 0}\|v(t)\| <\infty.$$ 
Clearly, $(u,v) \in \calU^*$ implies $(u,v)_{\{t \in [0,T]\}}\in \calU_T$. Since $\calU^*$ is an intersection of a decreasing sequence of non-empty compact sets $\{\calU_T:T>0\}$,  $\calU^*$ is a non-empty compact set. Hence there exists a $(u^*,v^*)\in \calU^*$ such that $(u^*,v^*)_{\{t \in [0,T]\}}$ is the optimizer for 
$$ \inf_{ (u,v) \in \calR_{x}(z,T)} \left( \widetilde V(z) +\frac{1}{2}\int_0^T \|u(s)\|^2 \D s +\gamma I_L(v)\right),$$
for every $T>0$. 
Consequently, for every $T>0$,
\begin{align} \label{e:vgtx}  V_{\gamma,T}(x)= \widetilde V(z) + \frac{1}{2}\int_0^T\|u^*(s)\|^2\D s + \gamma I_L(v^*)\,.\end{align} 

\medskip

\noindent {\bf (Step 3) Identifying the limiting optimal control problem:} \\

We know that $\gamma I_L(v^*)\in \{0,\gamma \alpha, 2\gamma \alpha,\ldots,p\gamma\alpha\}$ for every $T>0$. Then along a subsequence $T_k\to \infty$, 
$ \frac{1}{2}\int_0^{T_k}\|u^*(s)\|^2\D s +\gamma  I_L(v^*)$  converges
and therefore, $\widetilde V(z)$ also converges to say, $V^*$. Note that $z$ depends on $T$,  the dependence was suppressed in order to keep the notation simple. Now we make the dependence on $T$ explicit. Using compactness of the level sets of $\widetilde V$,  we can conclude that $z({T_k})$ converges to $z^*$ along a further sub-subsequence, which without any loss of generality, we denote by $T_k$ again. From the lower semicontinuity of $\widetilde V$, we get
\begin{align}\label{eq: hjb-inf-1} \widetilde V(z^*)\leq \liminf_{k\to \infty} \widetilde V(z(T_k))\doteq V^{*}\text{ and } \liminf_{k\to\infty} V_{\gamma,T_k} (x)= V^{*} + \frac{1}{2}\int_0^\infty\|u^*(s)\|^2\D s + \gamma I_L(v^*).\end{align} 
 
It only remains to show that $V^*=0$.  To do this, we again proceed along the lines of \cite[Lemma 6]{biswas2009small}. One difficulty straightway is that $\widetilde V$ a priori is known only to be lower continuous. Therefore we cannot infer that $\widetilde V(z^*)=V^{*}$.  To show this, we write~\eqref{eq-hjb} as
\begin{align*}
V_{\gamma,T} (x)= \inf_{y\in \RR^p} \inf_{ (u,v) \in{\calR}_{x}(y,T)} \left( \widetilde V(y) +\frac{1}{2}\int_0^{T} \|u(s)\|^2 \D s +\gamma I_L(v)\right).
\end{align*}
It is easy to see that for a fixed $y\in \RR^p$, the infimum 
$$H(x,y,T)\doteq   \inf_{ (u,v) \in{\calR}_x(y,T)} \left( \frac{1}{2}\int_0^T \|u(s)\|^2 \D s + \gamma I_L(v)\right)$$
can be handled exactly as before. Therefore, there exists a pair $(u^*,v^*)$  such that the solution $Y(\cdot)$ to~\eqref{ctdynR} satisfies $Y(T)\to y$ as $T\to\infty$. We also have 
$$ H(x,y)\doteq \lim_{T\to \infty} H(x,y,T)= \frac{1}{2}\int_0^\infty\|u^*(s)\|^2\D s + \gamma I_L(v^*).$$
\begin{center}and\end{center}
\begin{align}\label{eq:hjb-inf-2}
\liminf_{k\to\infty} V_{\gamma,T_k} (x)= \inf_{y\in \RR^p}\left(\widetilde V(y) + H(x,y)\right)\leq \widetilde V(z^*)+ H(x,z^*), \text{  }.
\end{align}
Here $z^*$ is as before. From~\eqref{eq: hjb-inf-1} and~\eqref{eq:hjb-inf-2}, we have
\begin{align*}
V^{*} + \frac{1}{2}\int_0^\infty\|u^*(s)\|^2\D s + \gamma I_L(v^*)&=\inf_{y\in \RR^p}\Big(\widetilde V(y) + H(z,y)\Big)\leq \widetilde V(z^*)+ H(z,z^*)\\
\implies V^{*}&\leq \widetilde V(z^*). 
\end{align*}
Here we have used the fact that $$ H(z,z^*)= \frac{1}{2}\int_0^\infty\|u^*(s)\|^2\D s + \gamma I_L(v^*).$$
Hence $V^{*}=\widetilde V(z^*)$.
From the foregoing, it is also clear that 
$$V_\gamma (x)= \widetilde V(z^*) + H(x,z^*) .$$
Now consider the corresponding trajectory $Y^*(t)$ which satisfies $Y^*(0)=x$ and $Y^*(t)\to z^*$ as $t\to \infty.$
It is clear that if $\widetilde V(z^*)>0$, then there exist $\delta,R$ such that $0<\delta<\|Y^*(t)\|<R$. Then from Lemma~\ref{lem:fleming-mod}, 
$$ V_\gamma (x)\geq \int_0^\infty \|u^*(s)\|^2\D s =\infty.$$
This is a contradiction, because $\widetilde V(x)<\infty$. Hence $\widetilde V(z^*)=0$ and we are done.
\end{proof}

{In the following result, we shall prove a locally uniformly upper bound result.
\begin{lemma} \label{lem-uc-local}Suppose $\{x_n:n\in \NN\}\subset \RR^p$ and $\{T_n:n\in \NN\}\subset \RR^+$ are sequences such that  $x^n\to x$, for some $x\in \RR^p$ and $T_n\uparrow \uparrow \infty$.  Then, we have 
$$ V_\gamma (x)\leq \liminf_{n\to\infty} V_{\gamma,T_n}(x_n)\,.$$
\end{lemma}
\begin{proof}
Fix sequences $\{x_n:n\in \NN\}$, $\{T_n:n\in \NN\}$ and $x$ as in the hypothesis of the lemma. For any $n$, from Step 2 of the proof of Proposition~\ref{prop-Vexp}(ii), we know that there exists a pair $(u^n,v^n)$ such that 
\begin{align*}   V_{\gamma,T_n}(x^n)= \widetilde V(z^n) + \frac{1}{2}\int_0^{T_n}\|u^n(s)\|^2\D s + \gamma I_L(v^n)\,.\end{align*} 
Recall that the  trajectory $Y^n$ associated with $(u^n,v^n)$ on $\RR^+$ is  defined according to~\eqref{ctdynR}--\eqref{ctdynR1} with $Y^n(0)=x^n$ and $Y^n(T_n)=z^n$. 

For $ \epsilon_n \doteq  \|x^n-x\|$,  it is clear that 
$$ \lim_{n\to\infty} \int_0^{\epsilon_n} \|u^n(s)\|^2\D s=0\,.$$

We now define another trajectory $\widetilde Y^n$ (which will be a sub-optimal trajectory with respect to $V_{\gamma,T_n}(x)$) as follows: $\widetilde Y^n$ is again defined according to~\eqref{ctdynR}--\eqref{ctdynR1} with $\widetilde Y^n(0)=x$ and $\widetilde Y(T_n)=z^n$ with control $(\widetilde u^n,\widetilde v^n)$ given by   $\widetilde v^n\equiv v^n$ and
\begin{align*}
\widetilde u^n(t)\doteq \begin{cases}\frac{Y^n(\epsilon_n)-x}{\epsilon_n} + b(\widetilde Y^n(t)),  & \text{ for $t <\epsilon_n$}\\
 u^n(t), & \text{ for $t\geq \epsilon_n$}\,.
\end{cases}
\end{align*}
From the construction, we can easily see that for $0\leq t<\epsilon_n$, $\widetilde Y^n(t)= x+ \frac{\big(Y^n(\epsilon_n)-x\big)}{\epsilon_n} t$ and  $\widetilde Y^n (t)=Y^n(t)$, for $t\geq \epsilon_n$.  Consequently, we can conclude that for $0\leq t<\epsilon_n $, 
$$\widetilde u^n(t)= \frac{Y^n(\epsilon_n)-x}{\epsilon_n}  + b\Big(x+ \frac{\big(Y^n(\epsilon_n)-x\big)}{\epsilon_n} t\Big)\,.  $$ Moreover, using a standard Gronwall argument and the fact that $x^n\to x$, we can also conclude that $Y^n(\epsilon_n)\to x$. Now, for $T=T_ n$ consider
\begin{align*}
V_{\gamma,T_n}(x)- V_{\gamma,T_n}(x_n) &\leq \int_0^{\epsilon_n} \| \widetilde u^n(s)\|^2\D s - \int_0^{\epsilon_n}\|u^n(s)\|^2 \D s\\
&\leq 3\frac{\|Y^n(\epsilon_n)-x\|^2}{\epsilon_n^2}\epsilon_n +3C \|x\|^2\epsilon_n + \frac{\|Y^n(\epsilon_n)-x\|^2}{\epsilon_n^2}C\epsilon^3_n - \int_0^{\epsilon_n}\|u^n(s)\|^2 \D s
\end{align*}
In the above, to get the second line, we use the fact that $\|b(x)\|\leq C\|x\|$, for some $C>0$ (this follows from the fact that $b$ is Lipschitz and $b(0)=0$).  Taking $n\to\infty$ and using the definition of $\epsilon_n$, the fact that $x_n\to x$ and  Proposition~\ref{prop-Vexp}(ii), we get 
\begin{align*}
V_{\gamma}(x)=\lim_{n\to\infty} V_{\gamma,T_n}(x) \leq \liminf_{n\to\infty} V_{\gamma,T_n}(x_n)\,.
\end{align*}
This completes the proof.
\end{proof}}

\begin{proof}[Proof of Theorem~\ref{thm-main}]
From Proposition~\ref{prop-finite} and the definition of WLDP, we know that for every compact set $K\subset \RR^p$
\begin{align}\label{eq-ub-T}
 \limsup_{n\to\infty} \frac{1}{\log(n)} \log \PP\big(X^{n,\gamma}(T) \in K\big)\leq -\inf_{x\in K} V_{\gamma,T}(x)
\end{align}
and for every open set $U\subset \RR^p$,
\begin{align}\label{eq-lb-T}
 \liminf_{n\to\infty} \frac{1}{\log(n)} \log  \PP\big(X^{n,\gamma}(T) \in U\big)\geq -\inf_{x\in U} V_{\gamma,T}(x)\,.
\end{align}
Here, the function $V_{\gamma,T}$ is as given by~\eqref{eq-hjb}. We first consider~\eqref{eq-ub-T}. Taking limit superior as $T\to\infty$ in~\eqref{eq-ub-T}, we obtain
\begin{align}\label{eq-ub-Ta}
\limsup_{T\to\infty} \limsup_{n\to\infty} \frac{1}{\log(n)} \log \PP\big(X^{n,\gamma}(T) \in K\big)\leq -\liminf_{T\to\infty}\inf_{x\in K} V_{\gamma,T}(x)\,.
\end{align}
Fix $\epsilon >0$. Let $x_T\in K$ be  $\epsilon$- optimal for $\inf_{x\in K} V_{\gamma,T}(x)$. Since $K$ is compact, for any sequence $T_n \rightarrow \infty$ , $x_{T_n}$ will have a  convergent subsequence $x_{T_{n_k}}$ that converges to some $x^*\in K$  as $k \rightarrow \infty$. From  Lemma~\ref{lem-uc-local}, we can infer that 

$$ \inf_{x\in K}V_{\gamma}(x)\leq V_{\gamma}(x^*)\leq \liminf_{k \to\infty} V_{\gamma,T_{n_k}}(x_{T_{n_k}})\leq \liminf_{k\to\infty} \inf_{x\in K} V_{\gamma,T_{n_k}}(x) +\epsilon.$$
  As the sequence was arbitrary, we have
  
$$ \inf_{x\in K}V_{\gamma}(x)\leq \liminf_{T\to\infty} \inf_{x\in K} V_{\gamma,T}(x) +\epsilon.$$  

  Taking $\epsilon \to 0$, along with~\eqref{eq-ub-Ta} we obtain \eqref{eq-ub}.\\

Now consider~\eqref{eq-lb-T} and take limit inferior as $T\to\infty$. This gives us
\begin{align}\label{eq-lb-Taxs
  }
\liminf_{T\to\infty} \liminf_{n\to\infty} \frac{1}{\log(n)} \log  \PP\big(X^{n,\gamma}(T) \in U\big)\geq -\limsup_{T\to\infty}\inf_{x\in U} V_{\gamma,T}(x)\,.
\end{align}
To prove \eqref{eq-lb} it is enough to show that 
\begin{equation}  \limsup_{T\to\infty}\inf_{x\in U} V_{\gamma,T}(x) \leq  \inf_{x\in U} V_{\gamma}(x) \,. \label{eq:cond} \end{equation}
To do this, we suppose the contrary \emph{i.e.,} there exists $\epsilon>0$ such that 
\begin{equation}  \inf_{x\in U} V_{\gamma}(x) +\epsilon\leq  \limsup_{T\to\infty}\inf_{x\in U} V_{\gamma,T}(x)\,. \label{eq:contra}
  \end{equation}
Let $x^*\in U$ be $\frac{\epsilon}{2}$-optimal for $\inf_{x\in U} V_{\gamma}(x) $. Then, we have
\begin{align*}
  V_{\gamma}(x^*)  +\frac{\epsilon}{2} &\leq \inf_{x\in U} V_{\gamma}(x) +\epsilon
\end{align*}
Using \eqref{eq:contra} and Proposition~\ref{prop-Vexp}(ii) we obtain
\begin{align*}
  V_{\gamma}(x^*)  +\frac{\epsilon}{2} &\leq  \limsup_{T\to\infty}\inf_{x\in U} V_{\gamma,T}(x)\\
&\leq  \limsup_{T\to\infty} V_{\gamma, T}(x^*)\\
&\leq  V_{\gamma}(x^*)\,.
\end{align*}
This is a contradiction and \eqref{eq:contra} does not hold. So \eqref{eq:cond} holds and the proof is complete.
\end{proof}

\appendix

\section{An optimal control result.}

Below we give a slightly modified version of \cite[Lemma 3.1]{fleming1977exit}. Under Assumption~\ref{main-assump}(i), it is clear that the equation 
\begin{align}\label{eq-rev}\dot{\phi}(t,x)= -b(\phi(t,x))\end{align} with $\phi(0,x)=x$
has a unique solution for all $t>0$ and $x\in \RR^p$.  Let $L>0$ be the  Lipschitz constant of $b(\cdot)$. From Assumption~\ref{main-assump}(ii), the fixed point $x=0$ is unstable in all directions (i.e. source) for the dynamics given by~\eqref{eq-rev}.

Suppose that the control $(u,v)$ is such that $v$ has $k$ jumps at $t_i, 1\leq i \leq k$, with $k \leq p$.  Then   $I_L(v)\leq p \gamma  \alpha$. Define a  trajectory $x(t)$ in $\RR^p$ given as follows: for $1\leq i\leq k$,
\begin{align} \dot{ x}(t)&= -b( x(t)) + u(t), \text{ for $t_{i-1}\leq t\leq t_i$} \label{app-1} \\\nonumber
	 x(t_i+)&= v(t_i)-v(t_i+)\,.  
	\end{align}
	with $x(t_k+)$ such that $r_1<\|x(t_k+)\|<r_2$, for some $0<r_1<r_2$.

\begin{lemma}\label{lem:fleming-mod}
	Suppose $x(\cdot)$ is as defined above in \eqref{app-1}. Then 
	$$ \int_0^\infty \|u(s)\|^2 \D s =+\infty\,.$$
\end{lemma}
\begin{proof} As mentioned already, from Assumption~\ref{main-assump}, it follows that the origin is unstable for \eqref{eq-rev} in all directions. Thus if  $\| z\|>r_1 $, there exists $T^*>0$  such that $\|\phi(t,z)\|\geq r_2$, for every $t>T^*$. We now argue that the $T^*$ above depends only on $r_1$ and not on $\|z\|$. For this, note that by the converse Lyapunov theorem (\cite{Sagar}, pp.\ 239-240), there exists a smooth\footnote{The aforementioned result does not claim smoothness of $V$, but that can be arranged by a suitable modification of the function as in \cite{Wilson}.} Lyapunov function $V : \R^p \to \R^+$ such that
  \begin{align}
    \label{eq:dagger} (i) & \langle\nabla V(x), b(x)\rangle < 0 \mbox{ for } x \neq 0,\mbox{ and,}\\
    &\nonumber \\
(ii)&\mbox{ there exist continuous increasing functions $\kappa_1, \kappa_2 : \RR^p \to \RR^+$ such that $\kappa_i(0) = 0,$}  \\& \mbox{$i = 1,2,$ with $\kappa_i(y) \uparrow\uparrow\infty, i = 1,2,$ as $y\uparrow\uparrow\infty$, such that }   k_1(\parallel x \parallel) \le V(\parallel x \parallel) \le k_2(\parallel x \parallel). \label{eq:daggerdagger}
\end{align}
It follows from \eqref{eq:dagger} that for $\|z\| > r_1$, any trajectory $\phi(\cdot, z)$ has
$$\frac{\D}{\D t}V(\phi(t,z)) = \langle-b(\phi(t,z)),\nabla V(\phi(t,z))\rangle$$
bounded away from $0$ uniformly in $\|z\| > r_1$ and hence $V(\phi(t,z))  \uparrow\uparrow \infty$. Therefore, by \eqref{eq:daggerdagger},  $\|\phi(t,z)\| \uparrow\uparrow \infty$ uniformly in $z$ whenever $\|z\|>r_1$. Hence $T^*$ depends only on $r_1$. Without loss in generality, we suppose that $T^*>t_k$. For  $t_k\leq t\leq  T^*$,
consider
$$\phi(t,x(t_k+))- x(t)= \int_{t_k}^t\big( \dot{x}(s) + b(x(s))\big)\D s + \int_{t_k}^t \big(-b(x(s) ) + b(\phi(s,x(t_k+)))\big)\D s. $$
By an application of triangle inequality followed by Jensen's inequality, we have
$$  \sup_{t_k\leq s\leq t}\|\phi(s,x(t_k+))- x(s)\| \leq (t-t_k)^{\frac{1}{2}} \sqrt{\int_{t_k}^t \|\dot{x}(s)+ b(x(s))\|^2\D s } + \int_{t_k}^t \|-b(x(s) ) + b(\phi(s,x(t_k+)))\|\D s.$$
Using the Lipschitz assumption on $b$ and by defintion of $u$,  we have 
  \begin{align*}
 \sup_{t_k\leq s\leq t}\|\phi(s,x(t_k+))- x(s)\| &\leq t^{\frac{1}{2}} \sqrt{\int_{t_k}^t \|u(s)\|^2\D s} + L\int_0^t \sup_{t_k\leq r\leq s}\|\phi(r,x(t_k+))- x(r) \|\D s\,.
\end{align*}
From Gr\"onwall's inequality,
\begin{align*}
r_2<\sup_{t_k\leq s\leq T^* }\|\phi(s,x(t_k+))- x(s)\|\leq \sqrt{T^*}e^{LT^*} \sqrt{\int_{t_k}^{T^*} \|u(s)\|^2\D s\,.}
\end{align*}
Since $T^*$ only depends on $r_1$, we may repeat above argument by replacing $x(t_k+)$ with  $x(T^*)$ to obtain
$$r_2< \sqrt{T^*}e^{LT^*} \sqrt {\int_{T^*}^{2T^*} \|u(s)\|^2\D s.\,}$$
 We may repeat this  inductively to obtain the result.
\end{proof}

\section*{Acknowledgements} SRA thanks the hospitality of Rice University during his tenure as a postdoctoral fellow  where  a significant part of this work was carried out. SA  acknowledges support of the Department of Atomic Energy, Government of India, under project no. RTI4001.   VSB thanks J.\ C.\ Bose and S.\ S.\ Bhatnagar Fellowships and the National Science Chair, all from the Government of India, for support during the course of this work. SA and VSB thank the generous hospitality of the International Centre for Theoretical Sciences, Bengaluru, where bulk of their work was done. The authors thank the Indian Institute of Technology Bombay for its hospitality, where part of this work was carried out.

		\bibliographystyle{plain}	
	\bibliography{references}
	\end{document}